\documentclass{amsart}

\usepackage{amssymb,stmaryrd}
\usepackage{amsmath}
\usepackage{graphicx}
\usepackage{epstopdf}
\newtheorem{theorem}{Theorem}[section]

\newtheorem{lemma}[theorem]{Lemma}

\newtheorem{corollary}[theorem]{Corollary}

\theoremstyle{definition}
\newtheorem{definition}[theorem]{Definition}

\theoremstyle{remark}

\begin{document}

\title[ A class of Integral Operators]{A class of Integral Operators from Lebesgue spaces into Harmonic Bergman-Besov or Weighted Bloch Spaces}

\author{\"{O}mer Faruk Do\u{g}an}
\address{Department of Mathematics,  Tek$\dot{\hbox{\i}}$rda\u{g} Nam{\i}k Kemal University,
59030 Tek$\dot{\hbox{\i}}$rda\u{g}, Turkey}
\email{ofdogan@nku.edu.tr}



\begin{abstract}
We consider a class of two-parameter weighted  integral operators induced by harmonic Bergman-Besov kernels on the unit ball of $\mathbb{R}^{n}$  and  characterize precisely those that are bounded from Lebesgue spaces $L^{p}_{\alpha}$ into Harmonic Bergman-Besov $b^{q}_{\beta}$ or weighted Bloch Spaces $b^{\infty}_{\beta} $, for  $1\leq p\leq\infty$, $1\leq q< \infty$ and $\alpha,\beta \in \mathbb{R}$.  These operators can be viewed as generalizations of the harmonic Bergman-Besov projections. Also, our results remove the disturbing conditions $\beta>-1$ when $q<\infty$ and $\beta\geq 0$ when $q=\infty$ of Do\u{g}an (A Class of Integral Operators Induced by Harmonic Bergman-Besov kernels on Lebesgue Classes, preprint, 2020) by mapping the operators into these spaces instead of the Lebesgue classes.
\end{abstract}

\date{\today}

\maketitle

\section{Introduction and Main Results}\label{section-Introduction}

Let $n\geq 2$ be an integer and $\mathbb{B}=\mathbb{B}_{n}$ be the open unit ball in $\mathbb{R}^n$. Let $\nu$  be the Lebesgue volume  measure on $\mathbb{B}$ normalized so that $\nu(\mathbb{B})=1$.
For  $\alpha\in \mathbb{R}$, we define the weighted volume measures $\nu_\alpha$ on $\mathbb{B}$ by
\[
d\nu_\alpha(x)=\frac{1}{V_\alpha} (1-|x|^2)^\alpha d\nu(x).
\]
These measures are finite when $\alpha>-1$ and in this case we choose $V_\alpha$ so that $\nu_\alpha(\mathbb{B})=1$. Naturally $V_0=1$.  For $\alpha\leq -1$, we set $V_\alpha=1$. We denote the Lebesgue classes with respect to $\nu_\alpha$ by $L^p_{\alpha}$,  $0<p<\infty$ and the corresponding norms by $\|\cdot\|_{L^p_{\alpha}}$.

Let $h(\mathbb{B})$ be the space of all complex-valued harmonic functions on $\mathbb{B}$ with the topology of uniform convergence on compact subsets. The space of bounded harmonic functions on $\mathbb{B}$ is denoted by $h^{\infty}$.
For $0<p<\infty$ and $\alpha>-1$, the weighted  harmonic Bergman space $b^p_\alpha$ is defined by $b^p_\alpha=  L^p_\alpha \cap h(\mathbb{B})$ endowed with the norm
$\|\cdot\|_{L^p_{\alpha}}$.  The subfamily $b^2_\alpha$ is a  reproducing kernel Hilbert space with respect to the inner product $[f,g]_{b^2_\alpha}=\int_{\mathbb{B}}f\overline{g} \, d\nu_{\alpha}(x)$ and with the reproducing kernel $R_\alpha(x,y)$  such that $f(x)=[f,R_\alpha(x,\cdot)]_{b^2_\alpha}$ for every $f\in b^2_\alpha$ and $x\in \mathbb{B}$. It is well-known that $R_\alpha$ is real-valued  and $R_\alpha(x,y)=R_\alpha(y,x)$. The homogeneous expansion of $R_\alpha(x,y)$ is given in the $\alpha>-1$ part of the formulas (\ref{Rq - Series expansion}) and (\ref{gamma k q-Definition}) below (see \cite{DS}, \cite{GKU2}).

For $\alpha>-1$, the orthogonal projection  $ Q_\alpha:L^2_\alpha \to b^2_\alpha$ is given by the integral operator
\begin{equation}\label{orthogonal projection}
  Q_\alpha f(x)= \frac{1}{V_\alpha} \int_{\mathbb{B}} R_\alpha(x,y) f(y) (1-|y|^2)^\alpha d\nu(y) \quad (f\in L^2_\alpha ).
\end{equation}
This integral operator plays a major role in the theory of weighted harmonic
Bergman spaces and the question when Bergman projection $ Q_\alpha:L^p_\beta \to b^p_\beta$  is bounded is studied in many sources such as 
(\cite[Theorem 3.1]{JP}, \cite[Theorem 2.5]{PE}, \cite[Theorem 3.1]{KS}). By
allowing  the exponents, the weights and  the parameters in the integrand  to be different, our goal  is
to determine exactly when the integral operator in (\ref{orthogonal projection}) is bounded from $L^p_\alpha$ to $b^q_\beta$.

Furthermore, we  also remove the restriction $\alpha > -1$. The weighted harmonic Bergman spaces $b^p_\alpha$ initially defined for $\alpha > -1$
and can be extended to the whole range $\alpha \in \mathbb{R}$. These are studied in detail in \cite{GKU2} and will be reviewed in Section \ref{section-Preliminaries}. We call the extended family $b^p_\alpha$ $(\alpha \in \mathbb{R}) $ as harmonic Bergman-Besov spaces and the corresponding reproducing kernels as $R_\alpha(x,y)$ $(\alpha \in \mathbb{R})$  harmonic Bergman-Besov kernels. The homogeneous expansion of  $R_\alpha(x,y)$ can be expressed in terms of zonal harmonics
\begin{equation}\label{Rq - Series expansion}
R_\alpha(x,y)=\sum_{k=0}^{\infty} \gamma_k(\alpha) Z_k(x,y) \quad  (\alpha\in \mathbb{R}, \, x,y\in \mathbb{B}),
\end{equation}
where (see \cite[Theorem 3.7]{GKU1}, \cite[Theorem 1.3]{GKU2})
\begin{equation}\label{gamma k q-Definition}
        \gamma_k(\alpha):= \begin{cases}
         \dfrac{(1+n/2+\alpha)_k}{(n/2)_k}, &\text{if $\, \alpha > -(1+n/2)$}; \\
         \noalign{\medskip}
         \dfrac{(k!)^2}{(1-(n/2+\alpha))_k (n/2)_k}, &\text{if $\, \alpha \leq -(1+n/2)$},
\end{cases}
\end{equation}
and $(a)_b$ is the Pochhammer symbol. For definition and details about $Z_k(x,y)$, see \cite[Chapter 5]{ABR}.

Finally, we allow the exponents $p, q$ to be $\infty$. We denote by $L^{\infty}=L^{\infty}(\nu)$  the Lebesgue class of all essentially bounded functions on $\mathbb{B}$ with respect to $\nu$. In this case we
have $L^{\infty}(d\nu_{\alpha})=L^{\infty}$ for every $\alpha \in \mathbb{R}$ and because of this we need to use a different
weighted class. For $\alpha\in \mathbb{R}$,  we define
\begin{equation*}
\mathcal{L}^{\infty}_\alpha := \{\varphi \,\, \text{is measurable on} \, \, \mathbb{B}: (1-|x|^2)^{\alpha} \varphi(x) \in L^{\infty} \},
\end{equation*}
so that $\mathcal{L}^{\infty}_0 = L^{\infty}$. The norm on $\mathcal{L}^{\infty}_\alpha$ is
\begin{equation*}
\|\varphi\|_{\mathcal{L}^{\infty}_\alpha} = \|(1-|x|^2)^\alpha \varphi(x)\|_{L^\infty}.
\end{equation*}
For $\alpha>0$, the weighted harmonic Bloch space $b_\alpha$ is defined by $h(\mathbb{B})\cap \mathcal{L}^\infty_\alpha$ and also can be extended to the whole range $\alpha \in \mathbb{R}$. The properties of this extended family are studied in detail in \cite{DU1} and will be reviewed in Section \ref{section-Preliminaries}.

We can now state our main result. For $b,c \in \mathbb{R}$  define the integral operator $T_{bc}$ by
\begin{equation}\label{Main Operator}
T_{bc}\, f(x) = \int_{\mathbb{B}} R_{c} (x,y) \, f(y) (1-|y|^2)^b d\nu(y).
\end{equation}
 Our aim is to determine exactly when $T_{bc}$ is bounded from $L^p_\alpha$ to $b^q_\beta$. The
result is divided into two cases depending on whether  $1\leq q< \infty$  or $q=\infty$ that describe its boundedness in terms of the six parameters $(b,c,\alpha,\beta,p,q)$ involved.

\begin{theorem}\label{Theorem-Boundedness of T 4}
Let  $b,c,\alpha,\beta \in \mathbb{R}$, $1\leq p\leq\infty$ and $1\leq q< \infty$.  Then
$T_{bc}$ is bounded from $L^p_\alpha$ to $b^q_\beta$ if and only if $(b,c,\alpha,\beta,p,q)$ satisfy one of the following conditions:
\begin{enumerate}
\item[(i)] $1<p\leq q<\infty$, $\alpha+1<p(b+1)$ and $c\leq b+\dfrac{n+\beta}{q}-\dfrac{n+\alpha}{p}$;
\item[(ii)] $1=p\leq q<\infty$, $\alpha<b$ and $c\leq b+\dfrac{n+\beta}{q}-(n+\alpha)$ or $\alpha\leq b$ and $c< b+\dfrac{n+\beta}{q}-(n+\alpha)$;
\item[(iii)] $1\leq q <p< \infty$, $\alpha+1<p(b+1)$ and $c< b+\dfrac{1+\beta}{q}-\dfrac{1+\alpha}{p}$;
 \item[(iv)] $1\leq q<p=\infty$, $\alpha-1<b$ and $c<b+\dfrac{\beta+1}{q}-\alpha$.
\end{enumerate}
\end{theorem}

\begin{theorem}\label{Theorem-Boundedness of T 5}
Let  $b,c,\alpha,\beta \in \mathbb{R}$ and $1\leq p\leq\infty$.  Then
$T_{bc}$ is bounded from $L^p_\alpha$ to $b^\infty_\beta$  if and only if $(b,c,\alpha,\beta,p)$ satisfy one of the following conditions:
\begin{enumerate}
\item[(i)] $1<p<\infty$, $\alpha+1<p(b+1)$ and $c\leq b+\beta-\dfrac{n+\alpha}{p}$;
\item[(ii)] $p=1$, $\alpha<b$ and $c\leq b+\beta-(n+\alpha)$ or $\alpha\leq b$ and $c< b+\beta-(n+\alpha)$;
\item[(iii)] $p=\infty$, $\alpha-1<b$ and $c\leq b+\beta-\alpha$.
\end{enumerate}
\end{theorem}

Moreover, we also determine when $T_{bc}$ is bounded from $L^p_\alpha$ to $h^\infty$.
\begin{theorem}\label{Theorem-Boundedness of T 6}
Let  $b,c,\alpha \in \mathbb{R}$ and $1\leq p\leq\infty$.  Then
$T_{bc}$ is bounded from $L^p_\alpha$ to $h^\infty$ if and only if $(b,c,\alpha,p)$ satisfy one of the following conditions:
\begin{enumerate}
\item[(i)] $1<p<\infty$, $\alpha+1<p(b+1)$ and $c< b-\dfrac{n+\alpha}{p}$;
\item[(ii)] $p=1$, $\alpha<b$ and $c\leq b-(n+\alpha)$ or $\alpha\leq b$ and $c< b-(n+\alpha)$;
\item[(iii)] $p=\infty$, $\alpha-1<b$ and $c< b-\alpha$.
\end{enumerate}
\end{theorem}

Harmonic Bergman-Besov projections on harmonic spaces  have been studied for some time. See \cite[Theorem 1.4]{GKU2} for  $1\leq p=q<\infty$, $b=c, \, \alpha=\beta \in \mathbb{R}$ and  \cite[Theorem 1.6]{DU1} for $p=q=\infty$, $b=c, \, \alpha=\beta \in \mathbb{R}$.   The other result we know of on Besov spaces is \cite[Theorem 4.1]{JP2} in which $1\leq p=q\leq\infty$, $b=c> -1$ and $\alpha=\beta=-n$. Note that in our results we have $b,c,\alpha,\beta \in \mathbb{R}$ unrestrictedly and thus  our operators in some sense generalize  the harmonic Bergman-Besov projections. The holomorphic counterparts of our results on the boundedness of  integral operators induced by holomorphic Bergman-Besov kernels  appear in \cite[Theorem 1.8 and Theorem 1.9]{KU1}.

To experts in analysis, one of the interesting problem might be the boundedness of $T_{bc}$
between different weighted Lebesgue classes. This problem is considered earlier in \cite{DOG1} as seven theorems that describe
boundedness of $T_{bc}$ in terms of the six parameters $(b,c,\alpha,\beta,p,q)$ involved and the proof of our main results most heavily depends on these results. We combine all the seven theorems out there as two theorems below depending on the value of $q$.  The following  theorem is a combination of  \cite[Theorems 1.1, 1.2, 1,3 and 1.4]{DOG1} with $1\leq q< \infty$. Note that they include an extra operator that replace $R_{c} (x,y)$ in the integral (\ref{Main Operator}) by  $|R_{c} (x,y)|$ because they need operators with positive kernels to apply Schur tests.

\begin{theorem}\label{Theorem-Boundedness of T 1}
Let  $b,c,\alpha,\beta \in \mathbb{R}$  with $\beta>-1$, $1\leq p\leq\infty$ and $1\leq q< \infty$.  Then
$T_{bc}$ is bounded from $L^p_\alpha$ to $L^q_\beta$ if and only if $(b,c,\alpha,\beta,p,q)$ satisfy one of the following conditions:
\begin{enumerate}
\item[(i)] $1<p\leq q<\infty$, $\alpha+1<p(b+1)$ and $c\leq b+\dfrac{n+\beta}{q}-\dfrac{n+\alpha}{p}$;
\item[(ii)] $1=p\leq q<\infty$, $\alpha<b$ and $c\leq b+\dfrac{n+\beta}{q}-(n+\alpha)$ or $\alpha\leq b$ and $c< b+\dfrac{n+\beta}{q}-(n+\alpha)$;
\item[(iii)] $1\leq q <p< \infty$, $\alpha+1<p(b+1)$ and $c< b+\dfrac{1+\beta}{q}-\dfrac{1+\alpha}{p}$;
 \item[(iv)] $1\leq q<p=\infty$, $\alpha-1<b$ and $c<b+\dfrac{\beta+1}{q}-\alpha$.
\end{enumerate}
\end{theorem}

The following  theorem is a combination of  \cite[Theorems 1.4, 1.5 and 1.7 ]{DOG1} with $ q=\infty$.

\begin{theorem}\label{Theorem-Boundedness of T 2}
Let  $b,c,\alpha,\beta \in \mathbb{R}$ with $\beta\geq0$ and $1\leq p\leq\infty$.  Then
$T_{bc}$ is bounded from $L^p_\alpha$ to $\mathcal{L}^\infty_\beta$ if and only if $(b,c,\alpha,\beta,p)$ satisfy one of the following conditions:
\begin{enumerate}
\item[(i)] $1<p<\infty$, $\alpha+1<p(b+1)$ and $c\leq b+\beta-\dfrac{n+\alpha}{p}$, and the strict inequality holds  when $\beta=0$;
\item[(ii)] $p=1$, $\alpha<b$ and $c\leq b+\beta-(n+\alpha)$ or $\alpha\leq b$ and $c< b+\beta-(n+\alpha)$;
\item[(iii)] $p=\infty$, $\alpha-1<b$ and $c\leq b+\beta-\alpha$, and the strict inequality holds when $\beta=0$.
\end{enumerate}
\end{theorem}

The conditions $\beta>-1$ when $q<\infty$ in Theorem \ref{Theorem-Boundedness of T 1}  and $\beta\geq 0$ when $q=\infty$  in Theorem \ref{Theorem-Boundedness of T 2} cannot be removed as  explained in \cite[Corollary 1.4]{DOG1}. Since we use repeatedly in this paper, this result is given again  as Corollary \ref{Remark-T,S with beta bigger than -1} below.   Notice that, our results are a variation of these theorems that removes the annoying conditions $\beta>-1$ when $q<\infty$ and $\beta\geq 0$ when $q=\infty$.

In the next  Section \ref{section-Preliminaries}, we collect some known facts about the harmonic Bergman-Besov and weighted Bloch spaces.
Sections \ref{Proof1} and \ref{Proof2} are devoted to the proofs of Theorem\ref{Theorem-Boundedness of T 4} and Theorems \ref{Theorem-Boundedness of T 5} and \ref{Theorem-Boundedness of T 6}, respectively.

\section{ Preliminaries}\label{section-Preliminaries}
For two positive expressions $X$ and $Y$, we  write $X\lesssim Y$  if there exists a positive constant  $C$, whose
exact value is inessential, such that $X\leq C Y$.
We also write $X\sim Y$  if both $X\lesssim Y$ and $Y\lesssim X$.

The Pochhammer symbol $(a)_b$ is defined by
\begin{equation*}
 (a)_b=\frac{\Gamma(a+b)}{\Gamma(a)},
\end{equation*}
when $a$ and $a+b$ are off the pole set $-\mathbb{N}$ of the gamma function. Stirling formula gives
\begin{equation}\label{Stirling}
  \frac{(a)_c}{(b)_c} \sim c^{a-b}, \quad c\to \infty.
\end{equation}

Let $\mathbb{S}$ be the unit sphere  in $\mathbb{R}^n$ and $\sigma$ be the   surface  measure on $\mathbb{S}$ normalized so that $\sigma(\mathbb{S})=1$. For $f\in L^{1}_{0}$, the polar coordinates formula is
\begin{equation*}
\int_{\mathbb{B}}f(x) \ d\nu(x)=n\int_{0}^{1} \epsilon^{n-1}\int_{\mathbb{S}}f(\epsilon\zeta) \ d\sigma(\zeta) \ d\epsilon,
\end{equation*}
in which $x=\epsilon\zeta$ with $\epsilon>0$ and $\zeta \in \mathbb{S}$.

We let $1\leq p,p'\leq \infty$ be the conjugate exponent. That is, if $1< p<\infty$, then $\frac{1}{p}+\frac{1}{p'}=1$; if $p=1$, then $p'=\infty$ and if $p=\infty$, then $p'=1$.

We show an integral inner product on a function space $X$  by $[\cdot,\cdot]_{X}$.

In multi-index notation, $m=(m_1,\dots,m_n)$ is an n-tuple of non-negative integers $m_1,\dots,m_n$ and
\begin{equation*}
 \partial^m f= \frac{\partial^{|m|} f}{\partial x_1^{m_1}\cdots\partial x_n^{m_n}}
\end{equation*}
is the usual partial derivative for smooth $f$, where $|m|=m_1+\dots+m_n$.

The weighted harmonic Bergman spaces $b^p_\alpha$ $(\alpha>-1)$ can be extended  to all $\alpha \in \mathbb{R}$. For $\alpha\in\mathbb{R}$ and $0<p<\infty$, let $N$ be a non-negative integer such that
$\alpha+pN>-1$. The harmonic Bergman-Besov space $b^p_\alpha$ consists of all $f\in h(\mathbb{B})$ such that
\[
(1-|x|^2)^N \partial^m f \in L^p_\alpha,
\]
for every multi-index $m$ with $|m|=N$. The space $b^p_\alpha$
 does not depend on the choice of $N$ as long as $\alpha+pN>-1$ is satisfied.

Likewise, given $\alpha\in \mathbb{R}$, pick a non-negative integer $N$ such that $\alpha+N>0$. The weighted harmonic Bloch space $b^\infty_\alpha$ consists of all $f\in h(\mathbb{B})$ such that
\[
 (1-|x|^2)^{N} \partial^m f \in \mathcal{L}^\infty_\alpha,
\]
for every multi-index $m$ with $|m|=N$. When $\alpha=0$, one can choose  $N=1$ and
\[
b^\infty_0=\Big\{f\in h(\mathbb{B}): \sup_{x\in\mathbb{B}}\, (1-|x|^2)|\nabla f(x)| <\infty\Big\}.
\]
This is the most studied member of the family. As before, the spaces $b^\infty_\alpha$
do not depend on the choice of $N$  as long as  $\alpha+N>0$
is satisfied

In the definitions of $b^p_\alpha$ and $b^\infty_\alpha$, instead of partial derivatives one can use  more effectively certain radial differential operators $D^t_s: h(\mathbb{B}) \to h(\mathbb{B})$, $(s,t \in \mathbb{R})$
defined in terms of reproducing kernels of harmonic Bergman spaces that are introduced in \cite{GKU1} and \cite{GKU2}.

Before going to the definition, note that for every $\alpha\in \mathbb{R}$ we have $\gamma_{0} (\alpha)=1$, and therefore
\begin{equation}\label{Rq(x,0)}
R_\alpha(x,0)=R_\alpha(0,y)=1, \quad (x,y\in \mathbb{B}, \alpha\in \mathbb{R}).
\end{equation}
Checking the two cases in (\ref{gamma k q-Definition}), we have by (\ref{Stirling})
\begin{equation}\label{gamma-k-asymptotic}
\gamma_k(\alpha) \sim k^{1+\alpha} \quad (k\to \infty).
\end{equation}
$R_\alpha(x,y)$ is harmonic as a function of either of its variables on $\overline{\mathbb{B}}$.

For any $f\in h(\mathbb{B})$ there exist homogeneous harmonic polynomials $f_k$ of degree $k$ such that
$f=\sum_{k=0}^{\infty} f_k$,  the series converging absolutely and uniformly on compact subsets of $\mathbb{B}$ which is called the homogeneous expansion of $f$ (see \cite{ABR}).

\begin{definition}
Let $f=\sum_{k=0}^\infty f_k\in h(\mathbb{B})$ be given by its homogeneous expansion. For $s,t\in\mathbb{R}$ we define  $D_s^t$ on $ h(\mathbb{B}) $  by
\begin{equation*}
  D_s^t f := \sum_{k=0}^\infty \frac{\gamma_k(s+t)}{\gamma_k(s)} \, f_k.
\end{equation*}
\end{definition}
By (\ref{gamma-k-asymptotic}), $\gamma_k(s+t)/\gamma_k(s) \sim k^t$ for any $s,t$ and, roughly speaking, $D_s^t$ multiplies the $k$th homogeneous part of $f$ by $k^{t}$. For every $s\in \mathbb{R}$, $D_s^0=I$, the identity. An important property of $D^t_s$ is that it is invertible with two-sided inverse $D_{s+t}^{-t}$:
\begin{equation*}
D^{-t}_{s+t} D^t_s = D^t_s D^{-t}_{s+t} = I,
\end{equation*}
which follows from the additive property $D_{s+t}^{z} D_s^t = D_s^{z+t}$.

For every $s,t \in \mathbb{R}$, the map $D^t_s: h(\mathbb{B})\to h(\mathbb{B})$ is continuous in the topology
of uniform convergence on compact subsets (see \cite[Theorem 3.2]{GKU2}). The parameter $s$ plays a minor role and is used to have the precise relation $D_s^t R_s(x,y)=R_{s+t}(x,y)$.

Consider the linear transformation $I_{s}^{t}$ defined for $f\in h(\mathbb{B})$ by
\begin{equation*}
  I^t_s f(x) := (1-|x|^2)^t D^t_s f(x).
\end{equation*}
The  spaces $b^p_\alpha$ and $b^\infty_\alpha$ can equivalently be defined by using the operators  $ D^t_s $.
\begin{definition}\label{definition of the h B-B space}
For $0<p<\infty$ and $\alpha \in \mathbb{R}$, we define the harmonic Bergman-Besov space $b^p_\alpha$ to consists of all $f\in h(\mathbb{B})$ for which $ I^t_s f$
belongs to  $L^p_\alpha$ for some  $s,t$ satisfying (see \cite{GKU2} when $1\leq p<\infty$, and \cite{DOG} when $0<p<1$)
\begin{equation}\label{alpha+pt}
 \alpha+pt>-1.
 \end{equation}
The quantity
\[
\|f\|^p_{b^p_\alpha} = \| I^t_s f\|^p_{L^p_\alpha}=c_\alpha \int_{\mathbb{B}} |D^t_s f(x)|^p (1-|x|^2)^{\alpha+pt} d\nu(x) <\infty
\]
defines a norm (quasinorm when $0<p<1$) on $b^p_\alpha$ for any such $s,t$.
\end{definition}

\begin{definition}\label{definition of the h B space}
For $\alpha \in \mathbb{R}$, we define the harmonic Bloch space $b^\infty_\alpha$ to consists of all $f\in h(\mathbb{B})$  for which $ I^t_s f$ belongs to  $\mathcal{L}^\infty_\alpha$ for some  $s,t$ satisfying (see \cite{DU1})
\begin{equation}\label{alpha+t}
 \alpha+t>0.
\end{equation}
The quantity
\[
\|f\|_{b^\infty_\alpha}=\| I^t_s f\|^p_{L^\infty_\alpha}= \sup_{x\in \mathbb{B}}\, (1-|x|^2)^{\alpha+t} |D^t_s f(x)| <\infty.
\]
defines a norm on $b^\infty_\alpha$ for any such $s,t$.
\end{definition}

 It is well-known that Definitions \ref{definition of the h B-B space} and \ref{definition of the h B space} are independent of $s,t$ under (\ref{alpha+pt}) and (\ref{alpha+t}), respectively and the norms (quasinorms when $0<p<1$) on a given space are all equivalent. Thus for a given pair $s,t$,  $ I^t_s$ isometrically imbeds $b^p_\alpha$ into $L^p_\alpha$ if and only if (\ref{alpha+pt}) holds, and $ I^t_s$ isometrically imbeds $b^\infty_\alpha$ into $L^\infty_\alpha$ if and only if (\ref{alpha+t}) holds.

The most significant  property of the operators $D^t_s$ is that it allows us to pass from one Bergman-Besov (or weighted Bloch) space to another. Actually, we have the following isomorphisms.

\begin{lemma}\label{Apply-Dst}
Let $0<p<\infty$ and $\alpha,s,t\in \mathbb{R}$.
\begin{enumerate}
  \item[(i)] The map $D^t_s:b^p_\alpha \to b^p_{\alpha+pt}$ is an isomorphism.
  \item[(ii)] The map $D^t_s:b^\infty_\alpha \to b^\infty_{\alpha+t}$ is an isomorphism.
\end{enumerate}
\end{lemma}

For a proof of part (i) of the above lemma see \cite[Corollary 9.2]{GKU2} when $1\leq p<\infty$ and \cite[Proposition 4.7]{DOG} when $0<p<1$. For part (ii) see \cite[Proposition 4.6]{DU1}.

The lemma below shows that  if $x$ stays close to $0$, then $R_\alpha(x,y)$ is uniformly away from $0$ for every $y\in \mathbb{B}$.
\begin{lemma}\label{Lemma-Stay away from 0}(\cite[Lemma 3.2]{DU1}).
Let $\alpha\in \mathbb{R}$. There exists $\epsilon>0$ such that for all $|x|<\epsilon$ and for all $y\in \mathbb{B}$, we have $R_\alpha(x,y) \geq 1/2$.
\end{lemma}

\section{Proof of Theorem \ref{Theorem-Boundedness of T 4}}\label{Proof1}

In this section, we prove Theorem \ref{Theorem-Boundedness of T 4}. Note that we call the second and third inequality in each of the four parts of Theorem \ref{Theorem-Boundedness of T 4}  the first and second necessary condition, respectively.

Before the proof, we formulate the behavior of the operators $T_{bc}$ in many important situations which will be used in this
and next section. These are from Section 3 of \cite{DOG1} and adapted from similar results in Section 4 of  \cite{KU1}.  We begin by inserting some obvious inequalities that we  use many times. If $a_{1}<a_{2}$, $u>0$, and $v \in \mathbb{R}$, then for $0\leq t<1$,
\begin{equation}\label{obvious-inequalities}
  (1-t^2)^{a_{1}}\leq (1-t^2)^{a_{2}} \qquad \text{and} \qquad (1-t^2)^{u} \big(1+\log (1-t^2)^{-1}\big)^{-v}\lesssim 1.
\end{equation}

The second inequality above leads to the following estimate.
\begin{lemma}\label{an estimate from calculus}(\cite[Lemma 3.1]{DOG1}).
For $u,v \in \mathbb{R}$,
\begin{equation*}
\int_{0}^{1} (1-t^{2})^{u}\big(1+\log \frac{1}{1-t^{2}}\big)^{-v} \, dt<\infty
\end{equation*}
if $u>-1$ or $u=-1$ and $v>1$, and the integral diverges otherwise.
\end{lemma}

We will use the functions
\begin{equation*}
f_{uv}(x)= (1-|x|^{2})^{u}\big(1+\log \frac{1}{1-|x|^{2}}\big)^{-v} \quad (u,v \in \mathbb{R})
\end{equation*}
as test functions to obtain some of the necessary conditions of our theorems from the action of $ T_{bc}$ on them.

\begin{lemma}\label{when fuv in Lpq}(\cite[Lemma 3.2]{DOG1}).
For $1\leq p<\infty$, we have $f_{uv}\in L^p_\alpha$ if and only if $\alpha+pu>-1$, or $\alpha+pu=-1$ and $pv>1$. For $p=\infty$, we have $f_{uv}\in \mathcal{L}^{\infty}_{\alpha}$ if and only if $\alpha+u>0$, or $u=-\alpha$ and $v\geq 0$.
\end{lemma}

\begin{lemma}\label{when Tfuv is finite} (\cite[Lemma 3.3]{DOG1}).
If $b+u>-1$ or if $b+u=-1$ and $v>1$, then $T_{bc}f_{uv}$ is a finite positive constant. Otherwise, $T_{bc}f_{uv}(x)=\infty$ for $|x|\leq \epsilon$, where $\epsilon$ is as in Lemma \ref{Lemma-Stay away from 0}.
\end{lemma}

\begin{proof}[Proof of Theorem \ref{Theorem-Boundedness of T 4}]
Let  $b,c,\alpha,\beta \in \mathbb{R}$, $1\leq p\leq\infty$ and $1\leq q< \infty$. \\ First Necessary Condition. Assume that $T_{bc}$ is bounded from $L^p_\alpha$ to $b^q_\beta$. The proof can be handled in three cases depending on the value of $p$.

 We first show the case $1<p<\infty$. Consider $f_{uv}$ with $u=-(1+\alpha)/p$ and $v=1$ so that $f_{uv}\in L^p_\alpha$ by Lemma \ref{when fuv in Lpq}. Then its clear that $T_{bc}f_{uv}\in b^q_\beta$ and this implies  $T_{bc}f_{uv}(0)\in \mathbb{C}$. We have by (\ref{Rq(x,0)})
\begin{equation*}
T_{bc}f_{uv}(0)=\int_{\mathbb{B}} (1-|x|^2)^{b-(1+\alpha)/p}\left(1+\log \dfrac{1}{(1-|x|)^{2}}\right)^{-1}d\nu(x).
\end{equation*}
Writing the integral in polar coordinates and using Lemma \ref{an estimate from calculus}, we obtain $(1+\alpha)/p<1+b$. Thus we derive the second inequality in parts (i) and (iii).

Next, we show the second case $p=1$. Consider $f_{uv}$ with $u>-(1+\alpha)$ and $v=0$ so that $f_{u0}\in L^1_\alpha$ by Lemma \ref{when fuv in Lpq}. Then $T_{bc}f_{u0}\in b^q_\beta$  and this implies  $T_{bc}f_{u0}(0)\in \mathbb{C}$. We have by (\ref{Rq(x,0)})
\begin{equation*}
T_{bc}f_{u0}(0)=\int_{\mathbb{B}} (1-|x|^2)^{b+u} d\nu(x)
\end{equation*}
 with $u>-(1+\alpha)$. Again writing  the integral in polar coordinates and using Lemma \ref{an estimate from calculus}, we obtain $\alpha\leq b$.
Thus we derive the second inequality in part (ii).

The last case is $p=\infty$. Let now  $f_{uv}$ with $u=-\alpha$ and $v=0$ so that $f_{u0}\in \mathcal{L}^\infty_\alpha$ by Lemma \ref{when fuv in Lpq}. Then $T_{bc}f_{u0}\in b^q_\beta$ and this implies $T_{bc}f_{u0}(0)\in \mathbb{C}$. We have by ( \ref{Rq(x,0)}) again,
\begin{equation*}
T_{bc}f_{u0}(0)=\int_{\mathbb{B}} (1-|x|^2)^{b-\alpha} d\nu(x).
\end{equation*}
 One more time writing  the integral in polar coordinates and using Lemma \ref{an estimate from calculus}, we obtain $b-\alpha>-1$.
Finally, we derive the second inequality in part (iv).

Now we will show the Second Necessary Condition. Having proved Theorems \ref{Theorem-Boundedness of T 1} and \ref{Theorem-Boundedness of T 2}, our job is easy with a technique of composing bounded maps. Firstly, we
present the following lemma which is a crucial component
of this technique and  will allow us to push $D^t_s$ into the integral operator $T_{bc}$.

\begin{lemma}\label{Lemma-Push-Dst into Tcb}
Let $c, b\in \mathbb{R}$ and $f\in L_b^1$. Then $D^t_c T_{bc}f=T_{b,c+t}f$ for every $t \in \mathbb{R}$.
\end{lemma}
\begin{proof}
Writing $T_{bc}f$ explicitly, it is easy to see that the proof can be verified in the same way as in  \cite[Lemma 2.3]{DU1}.  Thus we omit the details.
\end{proof}

We also need the following result of \cite{DOG1}.
\begin{corollary}\label{Remark-T,S with beta bigger than -1}(\cite[Corollary 3.6]{DOG1}).
If $T_{bc}:L^p_\alpha\to L^q_\beta$ is bounded and  $f\in L^p_\alpha$,  then $g=T_{bc}f$ is harmonic on $\mathbb{B}$. If also $q<\infty$, then   $\beta>-1$. Therefore $T_{bc}:L^p_\alpha\to b^q_\beta$  when it is bounded with $\beta>-1$ and $q<\infty$. Moreover, if  $\beta\leq -1$ and $q<\infty$, then  $T_{bc}:L^p_\alpha\to L^q_\beta$ is  not bounded. On the other hand, if $T_{bc}:L^p_\alpha\to L^\infty$ is bounded and  $f\in L^p_\alpha$,  then $g=T_{bc}f\in h^{\infty}$. Finally, If $T_{bc}:L^p_\alpha\to \mathcal{L}^\infty_{\beta}$ is bounded, $f\in L^p_\alpha$, and $\beta>0$  then $g=T_{bc}f\in b^\infty_{\beta}$. Moreover, if  $\beta< 0$, then  $T_{bc}:L^p_\alpha\to \mathcal{L}^\infty_\beta$ is  not bounded.
\end{corollary}

Second Necessary Condition. We assume that $T_{bc}$ is bounded and the first ne\-ces\-sary condition holds, and then we apply Theorem \ref{Theorem-Boundedness of T 1}. Let $f \in L^p_\alpha$. We first show also $f \in L^1_b$  in order to able to use Lemma \ref{Lemma-Push-Dst into Tcb}. In the  first  case $1< p<\infty$, by the H\"{o}lder inequality, we have
\begin{align*}
\|f\|_{L^1_b}&=\frac{V_{\alpha}}{V_{b}}\int_{\mathbb{B}} |f(x)|(1-|x|^2)^{b-\alpha}d\nu_{\alpha}(x)\\
&\lesssim \|f\|_{L_{\alpha}^{p}}\left(\int_{\mathbb{B}} (1-|x|^2)^{(b-\alpha)p'}d\nu_{\alpha}(x)\right)^{1/p'} <\infty,
\end{align*}
where the last inequality possible since $(b-\alpha)p'+\alpha=\frac{p(1+b)-(\alpha+p)}{p-1}> \frac{1+\alpha-(\alpha+p)}{p-1}=-1$ by the already obtained first necessary condition. In the  second  case $p=1$, we have $\alpha \leq b$ by the first necessary condition and thus $f \in L^1_b$ by (\ref{obvious-inequalities}). In the third case $p=\infty$, $f \in \mathcal{L}^\infty_\alpha$ and the first necessary condition gives $-1<b-\alpha$; thus
\begin{align*}
\|f\|_{L^1_b}&=\frac{1}{V_{b}}\int_{\mathbb{B}} (1-|x|^2)^{\alpha}|f(x)|(1-|x|^2)^{b-\alpha}d\nu(x)\\
&\lesssim \|f\|_{\mathcal{L}_{\alpha}^{\infty}}\left(\int_{\mathbb{B}} (1-|x|^2)^{b-\alpha}d\nu(x)\right) <\infty
\end{align*}
and $f\in L^1_b$ again.

Now consider the composition of bounded maps
 \begin{equation*}
 L^p_\alpha  \stackrel{\mathrm{T_{bc}}}{\xrightarrow{\hspace*{1cm}}}  b^q_\beta
 \stackrel{\mathrm{D_{c}^{-\beta/q}}}{\xrightarrow{\hspace*{1cm}}} b^q,
\end{equation*}
where Lemma \ref{Apply-Dst} (i) is put the work. With $L^p_\alpha \subset L^1_b$ at hand from the above arguments, this composition equals $T_{b,c-\beta/q}$ by Lemma \ref{Lemma-Push-Dst into Tcb}. We conclude that $T_{b,c-\beta/q}: L^p_\alpha \to L^q$ is bounded by Corollary \ref{Remark-T,S with beta bigger than -1}. Hence we obtain the third inequalities in all  parts  of the theorem with $c-\beta/q$ in place of $c$ and $0$ in place of $\beta$. But these are precisely the same
third inequalities in all  parts of the theorem.

Sufficiency. We assume that the three inequalities in all  parts of the theorem hold. Notice that the third inequality is equivalent to that with $c$ replaced by $c-\beta/q$ and   $\beta$ by $0$ such as in the above paragraph. Now Theorem \ref{Theorem-Boundedness of T 1} and Corollary \ref{Remark-T,S with beta bigger than -1} imply that $T_{b,c-\beta/q}: L^p_\alpha \to L^{q}\cap h(\mathbb{B})=b^q$ is bounded. Then the composition of maps
 \begin{equation*}
 L^p_\alpha  \stackrel{\mathrm{T_{b,c-\beta/q}}}{\xrightarrow{\hspace*{1cm}}}  b^q
 \stackrel{\mathrm{D_{c-\beta/q}^{\beta/q}}}{\xrightarrow{\hspace*{1cm}}} b^q_{\beta}
\end{equation*}
 is also bounded by Lemma \ref{Apply-Dst} (i). The second inequality in all parts of the theorem such as  in the proof of second necessity condition mentioned already yields $L^p_\alpha \subset L^1_b$. Thus by  Lemma \ref{Lemma-Push-Dst into Tcb}, this composition equals $T_{bc}: L^p_\alpha \to b^q_{\beta}$.
\end{proof}

\section{Proofs of Theorems \ref{Theorem-Boundedness of T 5} and \ref{Theorem-Boundedness of T 6} }\label{Proof2}

In this section, we prove Theorems \ref{Theorem-Boundedness of T 5} and \ref{Theorem-Boundedness of T 6}. Once more we  call the second and third inequality in each of the three parts of Theorems \ref{Theorem-Boundedness of T 5} and \ref{Theorem-Boundedness of T 6} the first and second necessary condition, respectively.

\begin{proof}[Proof of Theorem \ref{Theorem-Boundedness of T 5}]
Let $b,c,\alpha,\beta \in \mathbb{R}$ and $1\leq p \leq q=\infty$.
 First Necessary Condition. Assume that $T_{bc}$ is bounded from $L^p_\alpha$ to $b^\infty_\beta$. We  imitate the  proof of Theorem \ref{Theorem-Boundedness of T 4}. Thus we separate the proof in same three cases depending on the value of $p$ and use the same test functions $f_{uv}$ for each cases obtaining $T_{bc}f_{uv}\in b^{\infty}_{\beta}$. Then  $T_{bc}f_{uv}(0)\in \mathbb{C}$. Thus the second inequality in all parts of the theorem can be verified by repeating the first part of the above proof.

Second Necessary Condition. We assume that $T_{bc}$ is bounded and the first ne\-ces\-sary condition holds, and then we use Theorem \ref{Theorem-Boundedness of T 2}. Let $f \in L^p_\alpha$. With making small modifications in the second necessary part of above proof one can easily verify  $f \in L^1_b$ again.

Now consider the composition of bounded maps
 \begin{equation*}
 L^p_\alpha  \stackrel{\mathrm{T}_{bc}}{\xrightarrow{\hspace*{1cm}}}  b^\infty_\beta
 \stackrel{\mathrm{D_{c}^{-\beta+1}}}{\xrightarrow{\hspace*{1cm}}} b^\infty_{1},
\end{equation*}
where this time Lemma \ref{Apply-Dst} (ii) is put the work. With $L^p_\alpha \subset L^1_b$ at hand, the composition equals $T_{b,c-\beta+1}$ by Lemma \ref{Lemma-Push-Dst into Tcb}. We conclude that $T_{b,c-\beta+1}: L^p_\alpha \to \mathcal{L}^\infty_{1}$ is bounded by Corollary \ref{Remark-T,S with beta bigger than -1}. Hence we obtain the third inequalities in all parts of the theorem with $c-\beta+1$ in place of $c$ and $1$ in place of $\beta$. But these are precisely the same third inequalities in all parts of the theorem.

Sufficiency. We assume that the three inequalities in all parts of the theorem
hold. Notice that the third inequality is equivalent to that  with $c$ replaced by $c-\beta+1$ and   $\beta$ by $1$ such as in the above paragraph. Now Theorem \ref{Theorem-Boundedness of T 2} and Corollary \ref{Remark-T,S with beta bigger than -1} imply that $T_{b,c-\beta+1}: L^p_\alpha \to \mathcal{L}^\infty_{1}\cap h(\mathbb{B})=b^\infty_{1}$ is bounded. Then the composition of maps
 \begin{equation*}
 L^p_\alpha  \stackrel{\mathrm{T_{b,c-\beta+1}}}{\xrightarrow{\hspace*{1cm}}}   b^\infty_{1}
 \stackrel{\mathrm{D_{c-\beta+1}^{\beta-1}}}{\xrightarrow{\hspace*{1cm}}} b^\infty_{\beta}
\end{equation*}
 is also bounded by Lemma \ref{Apply-Dst} (ii). The second inequality in all parts of the theorem such as in
the proof of second necessity condition mentioned already  yields $L^p_\alpha \subset L^1_b$. Thus by  Lemma \ref{Lemma-Push-Dst into Tcb}, this composition equals $T_{bc}: L^p_\alpha \to b^\infty_{\beta}$ is bounded.
\end{proof}
\begin{proof}[Proof of Theorem \ref{Theorem-Boundedness of T 6}]

Assume that $T_{bc}:L^p_\alpha \to h^\infty$ is bounded. Then since   $h^\infty\subset L^{\infty}$ with the same norms,   $T_{bc}:L^p_\alpha \to L^\infty$  is also bounded. Therefore Theorem \ref{Theorem-Boundedness of T 2} implies the two necessary conditions in all three parts of the theorem hold.

Conversely, if both first and second necessary condition in  all three parts of the theorem hold, then  $T_{bc}:L^p_\alpha \to L^\infty$ is bounded by Theorem
\ref{Theorem-Boundedness of T 2}. But  Corollary \ref{Remark-T,S with beta bigger than -1}
shows that  the range of $T_{bc}$ lies in  $h^{\infty}$ rendering $T_{bc}:L^p_\alpha \to h^\infty$ is bounded.
\end{proof}

Finally, a question runs through one's  mind whether or not our results  on weighted spaces can be obtained from those on unweighted spaces with $\alpha=\beta=0$. It turns out that they can and we now explain how.  We assume that Theorems \ref{Theorem-Boundedness of T 4} and \ref{Theorem-Boundedness of T 5} are proved in the case $\alpha=\beta=0$ and we obtain them in the ge\-ne\-ral case with nonzero $\alpha,\beta$.  We go into details only for part (i) of Theorem \ref{Theorem-Boundedness of T 4}. We need some results from the previous sections whose proofs are independent of the proofs of our main results. Let $M_{v}$ denotes the operator of multiplication by $(1-|x|^{2})^{v}$. Obviously, $M_{-\alpha/p}: L^p \to L^p_\alpha $ is an  isomorphism.
Now, we will consider the sequence of the bounded maps
 \begin{equation*}
L^p  \stackrel{\mathrm{M_{-\alpha/p}}}{\xrightarrow{\hspace*{1cm}}} L^p_\alpha \stackrel{\mathrm{T_{bc}}}{\xrightarrow{\hspace*{1cm}}}  b^q_{\beta}
 \stackrel{\mathrm{D_{c}^{-\beta/q}}}{\xrightarrow{\hspace*{1cm}}} b^q,
\end{equation*}
where the last map is also an isomorphism by Lemma \ref{Apply-Dst} (i). The composition of these maps is $T_{b-\alpha/p,c-\beta/q}: L^p \to b^q $ by Lemma \ref{Lemma-Push-Dst into Tcb}, and its bounded if and only if the second map $T_{bc}: L^p_{\alpha} \to b^q_{\beta} $ is bounded. By assumption, the composition is bounded if and only if
\begin{equation*}
1<p(b+1)-\alpha \qquad \text{and} \qquad c-\frac{\beta}{q}\leq b-\frac{\alpha}{p}+n(\frac{1}{q}-\frac{1}{p}),
\end{equation*}
which are nothing but the inequalities in part (i) of Theorem \ref{Theorem-Boundedness of T 4}.

However, our proofs are not simplified significantly with  $\alpha=\beta=0$.
The classifications in Theorems \ref{Theorem-Boundedness of T 4}--\ref{Theorem-Boundedness of T 2} are according to $p, q$ and there seems to be no simple way of reducing them to fewer parts, because the inequalities in all seven parts can
change between $<$ and $\leq$ without any apparent reason with $p, q$.

\bibliographystyle{amsalpha}

\end{document}